\documentclass{amsart}
\newcommand{\Z}{\mathbb{Z}}
\newcommand{\N}{\mathbb{N}}
\newcommand{\R}{\mathbb{R}}
\newcommand{\ess}{\operatorname{\mathrm{ess}}}
\newtheorem{Theo}{Theorem}
\newtheorem{Def}{Definition}
\newtheorem{Cor}{Corollary}
\newtheorem{Prop}{Proposition}
\begin{document}
\title[Weakly multiplicative functions]{Weakly multiplicative arithmetic functions and the normal growth of groups}
\author[J.-C. Schlage-Puchta]{Jan-Christoph Schlage-Puchta}
\begin{abstract}
We show that an arithmetic function which satisfies some weak multiplicativity properties and in addition has a non-decreasing or $\log$-uniformly continuous normal order is close to a function of the form $n\mapsto n^c$. As an application we show that a finitely generated, residually finite, infinite group, whose normal growth has a non-decreasing or a $\log$-uniformly continuous normal order is isomorphic to $(\Z, +)$.
\end{abstract}
\maketitle
\section{Introduction and results}
A function $f:\N\rightarrow\R$ is called multiplicative, if for all coprime positive integers $n, m$ we have $f(nm)=f(n)f(m)$. P. Erd\H os \cite{Erdos} showed that a non-decreasing multiplicative function $f$ is of the form $f(n)=n^c$ for some $c\geq 0$. Birch \cite{Birch} showed that the same conclusion holds, if we assume that $f$ has a non-decreasing normal order (see Definition~\ref{def:analytic}). Following these results there has been a lot of activity dealing with similar statements for other regularity properties of multiplicative functions; however, the question whether ``multiplicative" can be replaced by a weaker statement has received much less attention. In \cite{Diplom} it was shown that a 
function $f$ is of the form $f(n)=n^c$ for some $c$, provided that $f$ has the following property: $f$ is monotonic, non-vanishing, and for all $n\in\N$ and all $\epsilon>0$ there is some $x_0>0$ such that for all $x>x_0$ the interval $[x, (1+\epsilon)x]$ contains some $m$ with $f(nm)=f(n)f(m)$. This statement was motivated by the fact that, if $G$ is a finitely generated group and if $f(n)$ denotes the number of normal subgroups of index $n$ in $G$, then $f$ satisfies some weak multiplicativity properties. In this note we will deal in a similar way with functions having a smooth normal order.

\begin{Def}
{\em A function $f:\N\rightarrow[0, \infty)$ is} weakly super-multiplicative{\em , if for all $n\in\N$ and all $\epsilon>0$ there exists some $x_0>0$ and some $\delta>0$ such that for all $x>x_0$ we have}
\[
\#\{m\in[x, (1+\epsilon)x]: f(nm)\geq (1-\epsilon)f(n)f(m)\} \geq \delta x.
\]
\end{Def}

Note that being weakly super-multiplicative is a very weak property. Clearly multiplicative functions are weakly super-multiplicative. A more striking example is the fact that if the values of $f(n)$ are chosen as the values of independent identically distributed random variables with values in $[0, 1]$, then $f$ is almost surely weakly super-multiplicative. To see this note that, as $f(m)\leq 1$ for all $m$, we have for every fixed $n$ that 
\[
\{m:f(nm)\geq f(n)f(m)\}\subseteq\{m:f(nm)\geq f(n)\}.
\]
Our claim now follows from the fact that for each $m$ the event $f(nm)\geq f(n)$ has positive probability.

\begin{Def}
\label{def:analytic}
\begin{enumerate}
\item {\em A function $f:\N\rightarrow[0, \infty)$ has} normal order $g${\em ,  if for all $\epsilon>0$ the set $\{n: |f(n)-g(n)|\geq \epsilon g(n)\}$ has upper density 0.}
\item {\em A function $g:(0, \infty)\rightarrow (0, \infty)$ is} $\log$-uniformly continuous{\em, if for every $\epsilon>0$ there exists some $\delta>0$ such that for all $x, y>0$ with $\left|\frac{x}{y}-1\right|<\delta$ we have $\left|\frac{g(x)}{g(y)}-1\right|<\epsilon$. }
\item {\em The} essential limit {\em $\lim\ess a_n$ of a sequence $(a_n)$ exists and is equal to $a$, if for all $\epsilon>0$ the set $\{n:|a_n-a|>\epsilon\}$ has density 0. We say the essential limit is $\infty$, if for all $M\in\R$ the set $\{n: a_n<M\}$ has density 0.}
\end{enumerate}
\end{Def}
Note that some authors include the monotonicity of $g$ in the definition of a normal order, however, we do not do so here.
With these notations we state the following.

\begin{Theo}
\label{thm:main}
Let $f$ be a weakly super-multiplicative function, which has a strictly positive normal order $g$, where $g$ is either non-decreasing or $\log$-uniformly continuous. Then
\[
\sup\frac{\log f(n)}{\log n} = \lim\ess\frac{\log f(n)}{\log n}.
\]
\end{Theo}

In particular $f(n)$ either tends super-polynomially to $\infty$, or it approaches $n^c$ for some constant $c$ from below. Note that a more precise statement is impossible, since for any function $\delta(n)$ which decreases monotonically to 0, the function $f(n)=n^{1-\delta(n)}$ is both strictly increasing and super-multiplicative, i.e. we have $f(nm)\geq f(n)f(m)$ for all $n, m$. This example shows that even if in Theorem~\ref{thm:main} we replace ``non-decreasing normal order" by ``strictly increasing", and ``weakly supermultiplicative" by ``super-multiplicative", the convergence to the limit can still be arbitrarily slow.

As a first application we recover a strengthening of Birch's result.

\begin{Cor}
\label{Cor:Birch}
Let $f:\N\rightarrow(0, \infty)$ be a function such that both $f$ and $f^{-1}$ are weakly super-multiplicative. If $f$ has a normal order that is monotonic or $\log$-uniformly continuous, then there is some $c$ such that $f(n)=n^c$ holds for all $n$.
\end{Cor}

As a second application we prove the following.

\begin{Cor}
\label{Cor:normal}
Let $G$ be a finitely generated residually finite group, and let $f(n)$ be the number of normal subgroups of $G$ of index $n$. If $f$ has a strictly positive normal order that is monotonic or $\log$-uniformly continuous, then $G\cong(\Z, +)$.
\end{Cor}

This result shows that the normal subgroup growth behaves completely different from subgroup growth. For the latter monotonicity has been established in a variety of cases, see e.g. \cite{free growth}, \cite{Tri}.

\section{Proof of the Theorem}

For the proof we first deduce a growth condition for $g$, given in equation (\ref{eq:g grows final}) below. The deduction of this condition depends on whether $g$ is supposed to be non-decreasing or $\log$-uniformly continuous. From that point onwards the proof of the two cases runs completely parallel.

{\bf A growth condition for monotonic $g$.} Let $n$ be an integer and $\epsilon>0$ a real number.  Let $x_0>0$ and $\delta>0$ be real numbers such that for $x>x_0$ we have  $f(nm)\geq(1-\epsilon)f(n)f(m)$ holds for $\geq\delta x$ integers  $m\in[x, (1+\epsilon)x]$. Let $x_1>0$ be a real number such that for $x>x_1$ we have that $|f(t)-g(t)|<\epsilon g(n)$ holds for all integers $t\in[x, (1+\epsilon)x]$ with at most $\frac{\delta}{3n} x$ exceptions. We conclude that for $x>\max(x_0, x_1)$ the interval $[x, (1+\epsilon)x]$ contains at least $\left(1-\frac{\delta}{3n}\right)x\geq \frac{2\delta}{3}x$ integers $m$ with
\[
f(nm)\geq(1-\epsilon)f(n)f(m)\geq (1-2\epsilon)f(n)g(m)\geq (1-2\epsilon)f(n)g(x),
\]
where in the last step we used the monotonicity of $g$. In the interval $[nx, n(1+\epsilon)x]$ there are at most $\frac{\delta}{3n}\cdot(nx)=\frac{\delta}{3}x$ integers $q$ with $|f(q)-g(q)|>\epsilon g(q)$, thus, for at least $\frac{\delta}{3}x$ integers $m\in[x, (1+\epsilon)x]$ we have 
\[
g(n(1+\epsilon)x)\geq g(nm)\geq(1-\epsilon)f(nm)\geq (1-3\epsilon)f(n)g(x)
\]
We conclude that for all $n$, all $\epsilon>0$ and all $x>x_0(n, \epsilon)$ we have 
\begin{equation}
\label{eq:g grows 1}
g(n(1+\epsilon)x)\geq(1-3\epsilon)f(n)g(x).
\end{equation}

{\bf A growth condition for $\log$-uniformly continuous $g$.} Let $n$ be an integer, $\epsilon>0$ be a real number, and let $0<\gamma\leq\epsilon$ be a real number such that $\left|\frac{x}{y}-1\right|<\gamma$ implies $\left|\frac{g(x)}{g(y)}-1\right|<\epsilon$. Let $x_0>0$ and $\delta>0$ be a real numbers such that for $x>x_0$ we have that $f(nm)\geq(1-\epsilon)f(n)f(m)$ holds for $\geq \delta x$ integers $m\in[x, (1+\gamma)x]$. As in the case $g$ non-decreasing we conclude that for $x$ sufficiently large we deduce
\[
g(nm)\geq(1-\epsilon)f(nm)\geq(1-\epsilon)^2f(n)f(m)\geq (1-\epsilon)^3f(n)g(m)
\]
for at least $\frac{\delta}{3}x$ integers $m\in[x, (1+\gamma)x]$. Using the fact that $g$ is $\log$-uniformly continuous and our definition of $\gamma$ we have for $m$ in this range the estimates $\left|\frac{g(nm)}{g((1+\gamma)nx)}-1\right|\leq\epsilon$ and $\left|\frac{g(m)}{g(x)}-1\right|<\epsilon$, thus
\begin{multline}
\label{eq:g grows 2}
g\big(n(1+\gamma)x\big)\geq\frac{1}{1+\epsilon}g(nm)  \geq\frac{(1-\epsilon)^3}{1+\epsilon}f(n)g(m)\\
\geq\frac{(1-\epsilon)^4}{1+\epsilon}f(n)g(x)\geq  (1-5\epsilon)f(n)g(x).
\end{multline}
{\bf Conclusion of the theorem.} Comparing (\ref{eq:g grows 1}) and (\ref{eq:g grows 2}) we find in either case that for every $n$ and every $\epsilon>0$ there exists some $\gamma$ in the range $0<\gamma\leq\epsilon$ and some $x_0=x_0(n, \epsilon)$ such that for $x>x_0$ we have
\begin{equation}
\label{eq:g grows final}
g(n(1+\gamma)x) \geq (1-5\epsilon)f(n)g(x).
\end{equation}

Iterating (\ref{eq:g grows final}) we obtain for $x>x_0(n, \epsilon)$ and an integer $k\geq 1$ the bound
\[
g(n^k(1+\gamma)^kx)\geq (1-5\epsilon)^k f(n)^k g(x).
\]
Put $\mu=\inf\{g(t):1\leq t\leq n(1+\gamma)\}$. If $g$ is non-decreasing, then $mu=g(1)$. If $g$ is $\log$-uniformly continuous, than in particular $g$ is continuous, thus $g$ attains its minimum in this interval. Since $g$ is strictly positive, in both cases we obtain $\mu>0$. Then we get for $y\in[n^k(1+\gamma)^k, n^{k+1}(1+\gamma)^{k+1}]$ the estimate
\[
g(y)\geq (1-5\epsilon)^k f(n)^k \mu,
\]
thus
\[
\liminf_{y\rightarrow\infty} \frac{\log g(y)}{\log y} \geq \liminf_{k\rightarrow\infty} \frac{\log \big((1-5\epsilon)^k f(n)^k m\big)}{\log\big(n^{k+1}(1+\gamma)^{k+1}\big)} =
\frac{\log \big((1-5\epsilon) f(n)\big)}{\log\big(n(1+\gamma)\big)}.
\]
As $\epsilon\rightarrow 0$, and $n$ ranges over all integers, we obtain $\liminf\frac{\log g(y)}{\log y}\geq \sup\frac{\log f(n)}{\log n}$. By the defition of a normal order we have 
\[
\limsup_{y\rightarrow\infty}\frac{\log g(y)}{\log y}\leq \sup\frac{\log f(n)}{\log n}\leq\liminf_{y\rightarrow\infty}\frac{\log g(y)}{\log y},
\]
thus $\lim\frac{\log g(y)}{\log y}$ exists and equals $\sup\frac{\log f(n)}{\log n}$. Again from the definition of the normal order we see that we can replace $\lim\frac{\log g(y)}{\log y}$ by $\lim\ess\frac{\log f(n)}{\log n}$, and the theorem follows.
\section{Proof of the Corollaries}

To prove Corollary~\ref{Cor:Birch} note that the conclusion of Theorem~\ref{thm:main} can be reformulated as stating that either $\lim\ess\frac{\log f(n)}{\log n}=\infty$, or there exists a constant $c$ and a non-negative function $\omega$, tending to 0, such that $f(n)\leq n^c$ holds for all $n$, and $f(n)=n^{c-\omega(n)}$ holds for almost all $n$. Hence, if $f$ and $f^{-1}$ are both weakly super-multiplicative, and $f$ has a strictly positive normal order which is either non-decreasing or $\log$-uniformly continuous, then there exist two constants $c_1, c_2$, and two non-negative functions $\omega_1, \omega_2$, tending to 0, such that $n^{c_1}\leq f(n)\leq n^{c_2}$ holds true for all $n$, and $n^{c_1+\omega_1(n)}=f(n)=n^{c_2-\omega_2(n)}$ holds for almost all $n$. But then $c_1+\omega_1(n)=c_2-\omega_2(n)$, since $\omega_i\rightarrow 0$, we deduce $c_1=c_2$ and $\omega_1(n)=\omega_2(n)=0$. This in turn is equivalent to the statement that $f(n)=n^{c}$ for all $n$.

To prove Corollary~\ref{Cor:normal} we first recall some properties of the number of normal subgroups of a finitely generated group.

\begin{Prop}
\label{prop:normal}
 Let $G$ be an $r$-generated group, $f(n)$ be the number of normal subgroups of index $n$.
\begin{enumerate}
\item If $(n,m)=1$, then $f(nm)\geq f(n)f(m)$.
\item For all $\epsilon>0$ we have that for almost all $n$ the inequality $f(n)\leq n^{r-1+\epsilon}$ holds.
\item If $n$ is an integer, $p$ a prime number, $(n, p(p-1))=1$, and $n$ has no non-trivial divisor $d\equiv 1\pmod{p}$, then $f(np)=f(n)$.
\end{enumerate}
\end{Prop}
\begin{proof}
The first statement follows from the fact that if $N, M$ are normal subgroups of $G$ of coprime index $m$ and $n$, then $M\cap N$ is a normal subgroup of index $mn$. Moreover, the map $(M,N)\mapsto M\cap N$ is injective, since in this case $G/(M\cap N)\cong (G/N)\times (G/M)$. The second statement is \cite[Theorem~2 (i)]{large}.

For the third statement let $H$ be a group of order $np$, where $n$ and $p$ satisfy the conditions of the proposition. By Sylow's theorem $H$ has a normal $p$ Sylow subgroup $P$, which is cyclic of order $p$. Hence, $h\in H$ acts on $P$ by conjugation. The order of $h$ divides $n$, and is therefore coprime to $|\mathrm{Aut}(C_p)|=p-1$, thus $h$ acts trivially on $P$. We conclude that $P$ is central in $H$. Since $(n,p)=1$, Zassenhaus' theorem implies that $P$ has a complement, and since $P$ is central, this complement is normal. We conclude that every group of order $np$ is the direct product of a group of order $n$ and a group of order $p$. This implies that in $G$ every normal subgroup of index $np$ is the intersection of a normal subgroup of index $n$ with a normal subgroup of index $p$, thus the map $(M,N)\mapsto M\cap N$ used to prove the first statement is actually a bijection, thus $f(np)=f(n)f(p)$.
\end{proof}

For an integer $n$, denote by $P^+(n)$ the largest prime divisor of $n$. Then we have  the following.

\begin{Prop}
\label{prop:divisors}
The set of integers $n$ such that $P^+(n)>\sqrt{n}$ and $(P^+(n)-1, n)=1$, has natural density $(\log 2)\prod\limits_p\left(1-\frac{1}{p(p-1)}\right)$.
\end{Prop}
\begin{proof}
We partition the set $\mathcal{A}$ of all integers $n\leq x$ with $P^+(n)>\sqrt{n}$ and $(P^+(n)-1, n)=1$ into three subsets, depending on the size of $P^+(n)$. Put
\begin{eqnarray*}
\mathcal{A}_1 & = & \{n\in\mathcal{A}: P^+(n)>\sqrt{x}\},\\
\mathcal{A}_2 & = & \{n\in\mathcal{A}: \frac{\sqrt{x}}{\log x}<P^+(n)\leq \sqrt{x}\},\\
\mathcal{A}_3 & = & \{n\in\mathcal{A}:P^+(n)\leq\frac{\sqrt{x}}{\log x}\}.
\end{eqnarray*}
As usual $\mathcal{A}_2$ and $\mathcal{A}_3$ are negligible, we therefore begin with estimating $|\mathcal{A}_1|$.

Fix a parameter $y$, and let $Q$ be the product of all prime numbers $\leq y$. Let $d$ be a divisor of $Q$. The Siegel-Walfisz-theorem implies that for $A$ fixed and $d<\log^A x$ we have
\[
\underset{p\equiv 1\pmod{d}}{\sum_{p\leq x}}\frac{1}{p} = \frac{1}{\varphi(d)}\log\log x + C_d + \mathcal{O}(\frac{1}{\log x}).
\]
Therefore the number of integers $n\leq x$ such that the largest prime divisor $p$ of $n$ is larger than $\sqrt{x}$, and $d|(n, p-1)$ equals
\begin{multline*}
\underset{p\equiv 1\pmod{d}}{\sum_{p\in[x^{1/2}, x]}} \#\{n\leq \frac{x}{p}:d|n\} = \underset{p\equiv 1\pmod{d}}{\sum_{p\in[x^{1/2}, x]}} \left(\frac{x}{dp}+\mathcal{O}(1)\right)\\
 = \frac{x}{d}\underset{p\equiv 1\pmod{d}}{\sum_{p\in[x^{1/2}, x]}}\frac{1}{p} + \mathcal{O}\left(\frac{x}{\log x}\right) = \frac{x}{d\varphi(d)}\log 2 + \mathcal{O}\left(\frac{x}{\log x}\right).
\end{multline*}

Since the product of all primes below $\log\log x$ is $(\log x)^{1+o(1)}$, this implies that  for $y\leq\log\log x$  the number of integers $n\leq x$ such that $P^+(n)>\sqrt{x}$ and $(n, P^+(n)-1, Q)=1$ is
\begin{eqnarray*}
\sum_{d|Q} \mu(d)\frac{x}{d\varphi(d)}\log 2 + \mathcal{O}(\frac{x}{\log x}) & = & x(\log 2)\prod_{p\leq y}\left(1-\frac{1}{p(p-1)}\right) + \mathcal{O}(\frac{\tau(Q)x}{\log x})\\
 & = & x(\log 2)\prod_{p\leq y}\left(1-\frac{1}{p(p-1)}\right) + \mathcal{O}(\frac{2^y x}{\log x})\\
\end{eqnarray*}
For modulus $d>\log^A x$ the prime number theorem for arithmetic progression might not hold anymore, we therefore switch to the Brun-Titchmarsh inequality in the form $\pi(x, q, a)\leq\frac{2x}{\varphi(q)\log(x/q)}$, which holds for all choices of $x$ and $q$. If $q\leq\sqrt[4]{x}$, we obtain by partial summation 
\begin{multline*}
\#\{n\leq x: P^+(n)>\sqrt{x}, q|(P^+(n)-1, n)\} = \underset{p\equiv1\pmod{q}}{\sum_{\sqrt{x}\leq p\leq x}} \left[\frac{x}{pq}\right]\\
\leq \frac{\pi(x, q, 1)}{xq} + \sum_{\sqrt{x}\leq t\leq x}\frac{\pi(t, q, 1)-\pi(\sqrt{x}, q, 1)}{qt(t-1)}
 \leq \frac{2x\log \sqrt{x}}{q(q-1)\log(\sqrt{x}/q)}
 \ll \frac{x}{q^2}.
\end{multline*}
For larger values of $q$ we omit the condition that $p$ be prime, and obtain similarly
\begin{multline*}
\#\{n\leq x: P^+(n)>\sqrt{x}, q|(P^+(n)-1, n)\} = \underset{\nu\equiv1\pmod{q}}{\sum_{\sqrt{x}\leq \nu\leq x}} \left[\frac{x}{q\nu}\right]\\
\leq \frac{x}{xq} + \sum_{\sqrt{x}\leq t\leq x}\frac{t-\sqrt{x}}{qt(t-1)}
 \leq \frac{2x\log x}{q(q-1)}
 \ll \frac{x\log x}{q^2}.
\end{multline*}

Merging these ranges we find that the number of integers $n\leq x$ such that $(P^+(n)-1, n)=1$ and $P^+(n)>\sqrt{x}$ equals
\begin{multline*}
x(\log 2)\prod_{p\leq y}\left(1-\frac{1}{p(p-1)}\right) + \mathcal{O}(\frac{2^y x}{\log x}) + \mathcal{O}\left(\sum_{y\leq q\leq \sqrt[4]{x}}\frac{x}{q^2}\right) + \mathcal{O}\left(\sum_{\sqrt[4]{x}\leq q\leq \sqrt{x}}\frac{x\log x}{q^2}\right)\\
= x(\log 2)\prod_{p\leq y}\left(1-\frac{1}{p(p-1)}\right) + \mathcal{O}(\frac{2^y x}{\log x}) + \mathcal{O}(\frac{x}{y})
\end{multline*}
For $y\geq 3$ we have 
\[
1>\prod_{p> y}\left(1-\frac{1}{p(p-1)}\right) \geq \exp\left(-\sum_{p>y} \frac{2}{p^2}\right)\geq \exp(-\frac{2}{y})\geq 1-\frac{2}{y},
\]
thus we can extend the product over all primes without enlarging the error term. Taking $y=\log\log x$ we obtain
\[
|\mathcal{A}_1| = x(\log 2) \prod_p\left(1-\frac{1}{p(p-1)}\right) + \mathcal{O}(\frac{x}{\log\log x}).
\]
Next we give upper bounds for $|\mathcal{A}_2|$ and $|\mathcal{A}_3|$. We have
\[
|\mathcal{A}_2| \leq \sum_{\frac{\sqrt{x}}{\log x}\leq p\leq \sqrt{x}}\left[\frac{x}{p}\right] \sim x\left(\log\log\sqrt{x} - \log\log\frac{\sqrt{x}}{\log x}\right)\sim\frac{2x\log\log x}{\log x}.
\]
Finally if $n\in\mathcal{A}_3$, then $\sqrt{n}\leq P^+(n)\leq\frac{\sqrt{x}}{\log x}$, thus $n\leq\frac{x}{\log^2 x}$, and therefore $|\mathcal{A}_3|\leq\frac{x}{\log^2 x}$.

We conclude that $|\mathcal{A}|\sim|\mathcal{A}_1|\sim x(\log 2)\prod\limits_p\left(1-\frac{1}{p(p-1)}\right)$, and our claim follows.
\end{proof}

To prove Corollary~\ref{Cor:normal}, note first that Proposition~\ref{prop:normal} (1) implies that we can apply Theorem~\ref{thm:main}. From Proposition~\ref{prop:normal} (2) we find that a normal order of $f$ grows at most polynomially, and conclude that there exists a constant $c$ and a non-negative function $\omega(n)$, tending to 0, such that $f(n)=n^{c-\omega(n)}$ for almost all $n$. 

If $n$ is an integer, $p$ the largest prime divisor of $n$, and $p>\sqrt{n}$, then $n/p$ has no divisor $d\neq 1$ that satisfies $d\equiv 1\pmod{p}$. If in addition $(n, p-1)=1$, then  Proposition~\ref{prop:normal} (3) implies $f(n)=f(n/p)f(p)$. Proposition~\ref{prop:divisors} shows that for a positive proportion of all integers $n$ we have $f(n)=f(n/P^+(n))f(P^+(n))$. Neglecting a set of integers $n$ of density 0 we may assume that $f(n)=n^{c-\omega(n)}$, and $f(n/p)=(n/p)^{c-\omega(n/p)}$. We obtain $f(p)=p^{c+o(1)}$ for infinitely many prime numbers $p$. On the other hand we know that every normal subgroup of prime index in $G$ contains the commutator of $G$, thus the number of normal subgroups of index $p$ in $G$ equals the number of subgroups of index $p$ in $G/G'$, where $G'$ is the commutator subgroup fo $G$. Being a finitely generated abelian group, this quotient is isomorphic to $A\oplus\Z^r$, where $A$ is some finite abelian group. Hence, for all but finitely many $p$ we have $f(p)=\frac{p^r-1}{p-1}=p^{r-1+o(1)}$. Comparing these two bounds we conclude that $c=r-1$. Hence, $\frac{p^r-1}{p-1}\leq f(p)\leq p^{r-1}$, which is only possible if $r=1$ and $A$ is trivial. We conclude that $f(n)\leq 1$ and $G/G'\cong\Z$. In particular, all normal subgroups of finite index contain $G'$. Since $G$ is residually finite, we conclude $G'=1$, and finally obtain $G\cong\Z$.

\end{document}